\documentclass[11pt,letterpaper]{amsart}

\usepackage{enumerate}
\usepackage{amsmath}
\usepackage{amsthm}
\usepackage{amssymb}
\usepackage{color}

\newtheorem{theorem}{Theorem}[section]
\newtheorem{corollary}[theorem]{Corollary}
\newtheorem{lemma}[theorem]{Lemma}
\newtheorem{conjecture}[theorem]{Conjecture}

\numberwithin{equation}{section}

\allowdisplaybreaks[3]
\DeclareMathOperator{\Tr}{Tr}
\DeclareMathOperator{\vol}{vol}
\newcommand{\C}{\mathbb{C}}
\newcommand{\F}{\mathbb{F}}
\newcommand{\N}{\mathbb{N}}
\newcommand{\R}{\mathbb{R}}
\newcommand{\Z}{\mathbb{Z}}
\newcommand{\e}{\epsilon}
\newcommand{\z}{\zeta}
\newcommand{\abs}[1]{\lvert#1\rvert}
\newcommand{\bigabs}[1]{\big\lvert#1\big\rvert}
\newcommand{\biggabs}[1]{\bigg\lvert#1\bigg\rvert}
\newcommand{\Biggabs}[1]{\Bigg\lvert#1\Bigg\rvert}
\newcommand{\leg}[2]{({#1}\!\mid\!{#2})}
\newcommand{\bigleg}[2]{\big({#1}\,\big|\,{#2}\big)}

\newcommand{\sums}[1]{\sum_{\substack{#1}}}
\newcommand{\floor}[1]{\lfloor{#1}\rfloor}
\newcommand{\ceil}[1]{\lceil{#1}\rceil}

\begin{document}

\title[Merit factor problem for binary sequences]{Advances in the merit factor problem for binary sequences}

\author{Jonathan Jedwab \and Daniel J.\ Katz \and Kai-Uwe Schmidt}

\date{03 May 2012 (revised 22 January 2013)}

\thanks{J.~Jedwab and D.J.~Katz are with Department of Mathematics, Simon Fraser University, 8888 University Drive, Burnaby BC V5A 1S6, Canada. K.-U.~Schmidt was with Department of Mathematics, Simon Fraser University and is now with Faculty of Mathematics, Otto-von-Guericke University, Universit\"atsplatz~2, 39106 Magdeburg, Germany.
Email: {\tt jed@sfu.ca}, {\tt dkatz@sfu.ca}, {\tt kaiuwe.schmidt@ovgu.de}.}
\thanks{J.~Jedwab is supported by NSERC}
\thanks{K.-U. Schmidt was supported by German Research Foundation.}


\begin{abstract}
The identification of binary sequences with large merit factor (small mean-squared aperiodic autocorrelation) is an old problem of complex analysis and combinatorial optimization, with practical importance in digital communications engineering and condensed matter physics. We establish the asymptotic merit factor of several families of binary sequences and thereby prove various conjectures, explain numerical evidence presented by other authors, and bring together within a single framework results previously appearing in scattered form. We exhibit, for the first time, families of skew-symmetric sequences whose asymptotic merit factor is as large as the best known value (an algebraic number greater than 6.34) for all binary sequences; this is interesting in light of Golay's conjecture that the subclass of skew-symmetric sequences has asymptotically optimal merit factor. Our methods combine Fourier analysis, estimation of character sums, and estimation of the number of lattice points in polyhedra.
\end{abstract}

\maketitle

\section{Introduction}
\label{sec:intro}
Let $A=(a_0,a_1,\dots,a_{t-1})$ be an element of $\{-1,1\}^t$ with $t>1$. We call $A$ a \emph{binary sequence of length $t$}. The \emph{aperiodic autocorrelation} of $A$ at shift $u$ is 
\[
c_u=\sum_{j=0}^{t-u-1}a_ja_{j+u}\quad\text{for $u\in\{0,1,\dots,t-1\}$}.
\]
Following Golay~\cite{Golay1972}, we define the \emph{merit factor} of $A$ to be
\[
F(A)=\frac{t^2}{2\sum_{u=1}^{t-1}c_u^2}.
\]
A large merit factor means that the sum of squares of the autocorrelations at nonzero shifts is small when compared to the squared autocorrelation at shift zero (which always equals $t^2$).
\par
The determination of the largest possible merit factor of long binary sequences is of considerable importance in various disciplines (see~\cite{Jedwab2005} and~\cite{MR2243493} for surveys, and~\cite{MR2436333} for background on related problems). In digital communications, binary sequences with large merit factor correspond to signals whose energy is very uniformly distributed over frequency~\cite{Beenker1985}. In theoretical physics, binary sequences achieving the largest merit factor for their length correspond to the ground states of Bernasconi's Ising spin model~\cite{Bernasconi1987}. The growth rate of the optimal merit factor of binary sequences, as the sequence length increases, is related to classical conjectures due to Littlewood~\cite{MR0196043},~\cite{MR0244463} and Erd\H{o}s~\cite[Problem~22]{MR0098702},~\cite{MR0141933},~\cite{MR1034349} on the asymptotic behavior of norms of polynomials on the unit circle.  This relationship arises because, when the binary sequence~$A$ is represented as a 
polynomial $A(z)=\sum_{j=0}^{t-1}a_jz^j$, its merit factor $F(A)$ satisfies
\begin{equation}
\frac{1}{F(A)}=-1+\frac{1}{2\pi t^2}\int_0^{2\pi}\bigabs{A(e^{i\theta})}^4\,d\theta    \label{eqn:F_int}
\end{equation}
(see \cite[pp.~370--371]{MR0196043} or~\cite[eq.~(4.1)]{Hoholdt1988}, for example).
\par
Littlewood~\cite[Chapter III, Problem 19]{MR0244463} proved in 1968 that the merit factor of Rudin-Shapiro sequences tends to $3$ as their length tends to infinity. H{\o}holdt and Jensen~\cite{Hoholdt1988}, building on studies due to Turyn and Golay~\cite{Golay1983}, proved in 1988 that the merit factor of Legendre sequences rotated by a quarter of their length is asymptotically~$6$, and conjectured that~$6$ is asymptotically the largest possible merit factor for binary sequences. 
But the present authors~\cite{Jedwab2011} recently disproved this conjecture by showing that a certain family of binary sequences attains an asymptotic merit factor $F_a=6.342061\dots$, which is the largest root of $29x^3-249x^2+417x-27$.
These sequences, called {\it appended rotated Legendre sequences}, had been studied numerically by Kirilusha and Narayanaswamy~\cite{Kirilusha-Narayanaswamy} and Borwein, Choi, and Jedwab~\cite{MR2103494}.
\par
Prior to the paper~\cite{Jedwab2011}, only two methods were known for calculating the asymptotic merit factor of a family of binary sequences~\cite{MR2243493}. The first is direct calculation, particularly in the case that the polynomials are recursively defined~\cite{MR0244463}. The second, introduced by H{\o}holdt and Jensen~\cite{Hoholdt1988} in 1988, is more widely applicable~\cite{MR1008545}, \cite{MR1145821}, \cite{MR1766099}, \cite{MR1814964}, \cite{MR1912495}, \cite{MR1859033}, \cite{MR2506407}, \cite{MR2648554}, \cite{Jedwab2012}.
The new approach of \cite{Jedwab2011} made it possible for the first time to handle appended rotated Legendre sequences, thereby showing that an asymptotic merit factor of $6$ can be exceeded.
In this paper, we elaborate and further develop the method of~\cite{Jedwab2011} to deal with other highly-studied binary sequence families, including Galois sequences (also known as m-sequences), Jacobi sequences, and sequences formed using Parker's periodic and negaperiodic constructions~\cite{MR1913466}.
This allows us to explain several previous numerical results and prove a series of conjectures \cite{Parker2005}, \cite{MR2629545}, \cite{MR2444577}, \cite{Jedwab2010} (see Section~\ref{sec:previous}). Moreover, we give simple unifying proofs, as well as generalizations, of the main results of~\cite{Hoholdt1988}, \cite{MR1008545}, \cite{MR1145821}, \cite{MR1913466}, \cite{MR2103494}, \cite{MR2444577}, \cite{MR2506407}, \cite{MR2648554}, \cite{Jedwab2012} and~\cite{Jedwab2011}.
\par
The binary sequences we consider in this paper fall into two classes. The largest achievable asymptotic merit factor for the first class, based on Legendre sequences, is $F_a=6.342061\dots$ mentioned above, whereas that for the second class, based on Galois sequences, is $F_b=3.342065\dots$, the largest root of $7x^3-33x^2+33x-3$.
\par
A binary sequence $(a_0,a_1,\ldots,a_{2s})$ of odd length $2s+1$ is called \emph{skew-symmetric} if
\[
a_{s+j}=(-1)^ja_{s-j}\quad\text{for all $j\in\{1,2,\dots,s\}$}.
\]
Historically, skew-symmetric binary sequences have been considered good candidates for a large merit factor (see \cite[Section~3.1]{Jedwab2005} for background), in part because half of their aperiodic autocorrelations are zero~\cite{Golay1972}.  Computer calculations indicate~\cite[Table III]{Golay1977}, \cite{Mertens2001} that skew-symmetric binary sequences have largest possible merit factor among all binary sequences of their length, for all odd lengths between $2$ and $60$ except $19$, $23$, $25$, $31$, $33$, $35$, and $37$.
Golay conjectured \cite{Golay1977},~\cite{Golay1982}, based on a heuristic argument, that the largest asymptotic merit factor among all binary sequences is attained by skew-symmetric sequences. It is interesting, in light of Golay's conjecture, that Corollary~\ref{cor:main_3} provides the first known families of skew-symmetric binary sequences with asymptotic merit factor~$F_a=6.342061\dots$.
\par
To the authors' knowledge, this paper contains all currently known results on the asymptotic merit factor of nontrivial families of binary sequences, except for Rudin-Shapiro sequences~\cite{MR0244463} and related binary sequence families~\cite{Hoholdt1985},~\cite{MR1709147}, and certain modifications of Jacobi sequences \cite{Jedwab2012},~\cite{Xiong2011},~\cite{Xiong2011a}.

\section{Results}

Let $A(z)=\sum_{j=0}^{n-1}a_jz^j$ be a polynomial of degree $n-1$ with coefficients in $\{-1,1\}$; we call $(a_0,a_1,\dots,a_{n-1})$ the \emph{coefficient sequence} of $A$, and write $F(A)$ for its merit factor. Let $r$ and $t$ be integers that can depend on~$n$, where $t \ge 0$, and define the polynomial
\[
A^{r,t}(z)=\sum_{j=0}^{t-1}a_{j+r} z^j,
\]
where henceforth we extend the definition of $a_j$ so that $a_{j+n}=a_j$ for all $j \in\Z$.
The coefficient sequence of $A^{r,t}$ is derived from that of $A$ by cyclically permuting (rotating) the sequence elements through $r$ positions, and then truncating when $t<n$ or periodically extending (appending) when $t>n$. 
We follow Parker~\cite[Lemma~3]{MR1913466} by applying a ``negaperiodic'' construction to~$A$ to give the polynomial 
\[
N(A)(z)=\sum_{j=0}^{4n-1}(-1)^{j(j-1)/2}a_j z^j,
\]
whose coefficient sequence is the element-wise product of the coefficient sequence of $A^{0,4n}$ with the sequence $(+,+,-,-,+,+,-,-,\dots,+,+,-,-)$ of length $4 n$. We also follow Parker~\cite[Lemma~4]{MR1913466} by applying a ``periodic'' construction to~$A$ to give the polynomial
\[
P(A)(z)=\sum_{j=0}^{4n-1}(-1)^{j(j-1)^2/2}a_j z^j,
\]
whose coefficient sequence is the element-wise product of the coefficient sequence of $A^{0,4n}$ with the sequence $(+,+,-,+,+,+,-,+,\dots,+,+,-,+)$ of length $4 n$.\footnote{Our constructions are cyclically permuted versions of those of Parker \cite{MR1913466}, and our $N(A)$ is defined to be twice as long as Parker's; we address all cyclic shifts and lengths in our results, but the definitions above give the most convenient reference point for subsequent calculations.} The advantage of interpreting Parker's constructions in terms of product sequences was recognized by Xiong and Hall~\cite{MR2444577} in the negaperiodic case, and by Yu and Gong \cite{MR2629545} in the periodic case.
\par
Let $p$ be an odd prime. The \emph{Legendre symbol} $\leg{j}{p}$ is given by
\[
\leg{j}{p}=\begin{cases}
0  & \text{if $j\equiv 0\pmod p$},\\
-1 & \text{if $j$ not a square modulo $p$},\\
+1 & \text{otherwise},
\end{cases}
\]
and the coefficient sequence of
\begin{equation}\label{LegendreSequenceDefinition}
X_p(z)=1+\sum_{j=1}^{p-1}\bigleg{j}{p}\,z^j   
\end{equation}
is a binary sequence called the \emph{Legendre sequence} of length~$p$.
\par
Define the function $g:\R\times \R^+\to\R$ by
\[
\frac{1}{g(R,T)}=1-\frac{4T}{3}+4\sum_{m\in\N}\max\bigg(0,1-\frac{m}{T}\bigg)^2+\sum_{m\in\Z}\max\bigg(0,1-\biggabs{1+\frac{2R-m}{T}}\bigg)^2,
\]
where $\N$ is the set of positive integers.
Then we have the following asymptotic merit factor result for Legendre sequences, and their negaperiodic and periodic versions.
\begin{theorem}
\label{thm:main_1}
Let $X_p$ be the Legendre sequence of length $p$ and let $R$ and $T>0$ be real. Then the following hold, as $p\to\infty$:
\begin{enumerate}[(i)]
\setlength{\itemsep}{1ex}
\item \label{itm:11} If $r/p\to R$ and $t/p\to T$, then $F(X_p^{r,t})\to g(R,T)$.
\item \label{itm:12} If $r/(2p)\to R$ and $t/(2p)\to T$, then $F(N(X_p)^{r,t})\to g(R+\frac14,T)$.
\item \label{itm:13} If $r/(4p)\to R$ and $t/(4p)\to T$, then $F(P(X_p)^{r,t})\to g(R,T)$.
\end{enumerate}
\end{theorem}
\par
Theorem~\ref{thm:main_1}~\eqref{itm:11} is the main result of~\cite{Jedwab2011}. The function $g$ satisfies $g(R,T)=g(R+\frac12,T)$ on its entire domain. As shown in~\cite[Corollary 3.2]{Jedwab2011}, the global maximum of $g(R,T)$ exists and equals 
\begin{equation}\label{eqn:Fa}
\mbox{$F_a=6.342061\dots$, the largest root of $29x^3-249x^2+417x-27$.}
\end{equation}
The global maximum is unique for $R\in[0,\frac12)$, and is attained when $T=1.057827\dots$ is the middle root of $4x^3-30x+27$ and $R=\tfrac34-\tfrac{T}{2}$. 
\par
Now let $\F_{2^d}$ be the finite field with $2^d$ elements and write $n=2^d-1$. Let $\psi:\F_{2^d}\to\{-1,1\}$ be the canonical additive character of $\F_{2^d}$, given by
\[
\psi(y)=(-1)^{\Tr(y)},
\]
where $\Tr(y)=\sum_{j=0}^{d-1}y^{2^j}$ is the absolute trace on $\F_{2^d}$.
Let $\theta$ be a primitive element of $\F_{2^d}$ and define the polynomial 
\begin{equation}\label{GaloisSequenceDefinition}
Y_{n,\theta}(z)=\sum_{j=0}^{n-1}\psi(\theta^j)\,z^j.
\end{equation}
The coefficient sequence of $Y_{n,\theta}$ is a binary sequence which we call the \emph{Galois sequence} of length~$n$ with respect to $\theta$ (cf. \cite{Schroeder} for this terminology).\footnote{The \emph{m-sequences} associated with~$\theta$ are the $n$ cyclic permutations of this Galois sequence. Their corresponding polynomials are $Y_{n,\theta}^{r,n}$ for $r=0,1\dots,n-1$, all of which we handle in Theorem~\ref{thm:main_2}.}
\par
Define the function $h:\R^+\to\R$ by
\[
\frac{1}{h(T)}=1-\frac{2T}{3}+4\sum_{m\in\N}\max\bigg(0,1-\frac{m}{T}\bigg)^2.
\] 
Then we have the following asymptotic merit factor result for Galois sequences, and their negaperiodic and periodic versions.
\begin{theorem}
\label{thm:main_2}
For each $n=2^d-1$, choose an integer~$r$ and a primitive $\theta\in\F_{2^d}$, and let $Y_{n,\theta}$ be the Galois sequence of length~$n$ with respect to~$\theta$. Let $T>0$ be real. Then the following hold, as $n\to\infty$:
\begin{enumerate}[(i)]
\setlength{\itemsep}{1ex}
\item \label{itm:21} If $t/n\to T$, then $F(Y_{n,\theta}^{r,t})\to h(T)$.
\item \label{itm:22} If $t/(2n)\to T$, then $F(N(Y_{n,\theta})^{r,t})\to h(T)$.
\item \label{itm:23} If $t/(4n)\to T$, then $F(P(Y_{n,\theta})^{r,t})\to h(T)$.
\end{enumerate}
\end{theorem}
\par
Elementary calculus shows that $h(T)$ is strictly decreasing on the intervals $[2,3]$, $[3,4],\ldots$, and so one can confine the optimization problem to $[0,2]$, where it is not hard to show that the global maximum of $h(T)$ is unique and is attained for $T=1.115749\dots$, which is the middle root of $x^3-12x+12$. The maximum value attained there is 
\[
\mbox{$F_b=3.342065\dots$, the largest root of $7x^3-33x^2+33x-3$}.
\]
We find it rather curious that, if $(R_a,T_a)$ is the pair $(R,T)$ that maximizes $g(R,T)$ and $T_b$ is the $T$ that maximizes $h(T)$, then the algebraic numbers
\begin{align*}
g(R_a,T_a)-6&=0.342061\dots
\intertext{and}
h(T_b)-3&=0.342065\dots
\end{align*}
are distinct, but first differ in only the sixth decimal place. Likewise, the algebraic numbers
\begin{align*}
T_a-1&=0.057827\dots
\intertext{and}
\tfrac{1}{2}(T_b-1)&=0.057874\dots
\end{align*}
are distinct, but first differ in only the fifth decimal place.
\par
Our third main result is a far-reaching generalization of Theorem~\ref{thm:main_1}. For $j$ an integer and $n$ a positive odd integer, the \emph{Jacobi symbol} $\leg{j}{n}$ extends the Legendre symbol via $\leg{j}{1}=1$ and $\leg{j}{n_1}\leg{j}{n_2}=\leg{j}{n_1n_2}$ for positive odd integers $n_1, n_2$. For $n$ a positive odd square-free integer, the coefficient sequence of
\[
X_n(z)=\sum_{j=0}^{n-1}\bigleg{j}{\tfrac{n}{\gcd(j,n)}}\,z^j   
\]
is a binary sequence called the \emph{Jacobi sequence} of length~$n$. When $n$ is prime, then $X_n$ is the Legendre sequence of length $n$. 
\par
We denote by $\omega(n)$ and $\kappa(n)$ the number of distinct prime divisors of $n$ and the smallest prime divisor of $n$, respectively. Then the merit factor for Jacobi sequences, and their negaperiodic and periodic versions, has the same asymptotic form as that for Legendre sequences as presented in Theorem~\ref{thm:main_1}.
\begin{theorem}
\label{thm:main_3}
Let $n>1$ take values in an infinite set of positive odd square-free integers such that
\begin{equation}
\frac{\max(4^{\omega(n)}(\log n)^6,5^{\omega(n)})}{\kappa(n)}\to0\quad\text{as $n\to\infty$}.
\label{eqn:cond_product}
\end{equation}
Let $X_n$ be the Jacobi sequence of length $n$ and let $R$ and $T>0$ be real. Then the following hold, as $n\to\infty$.
\begin{enumerate}[(i)]
\setlength{\itemsep}{1ex}
\item \label{itm:31} If $r/n\to R$ and $t/n\to T$, then $F(X_n^{r,t})\to g(R,T)$.
\item \label{itm:32} If $r/(2n)\to R$ and $t/(2n)\to T$, then $F(N(X_n)^{r,t})\to g(R+\tfrac{1}{4},T)$.
\item \label{itm:33} If $r/(4n)\to R$ and $t/(4n)\to T$, then $F(P(X_n)^{r,t})\to g(R,T)$.
\end{enumerate}
\end{theorem}
\par
In the special case where each $n$ is prime, Theorem~\ref{thm:main_3} reduces to Theorem~\ref{thm:main_1}.
\par
The following corollary is an immediate consequence of Theorem~\ref{thm:main_3}, and the fact that $\leg{j}{d}=\leg{-j}{d}$ when $d\equiv 1\pmod 4$. 
\begin{corollary}
\label{cor:main_3}
Let $n>1$ take values in an infinite set of positive odd square-free integers such that each prime divisor of $n$ is congruent to $1$ modulo $4$ and such that
\[
\frac{\max(4^{\omega(n)}(\log n)^6,5^{\omega(n)})}{\kappa(n)}\to0\quad\text{as $n\to\infty$}.
\]
Let $X_n$ be the Jacobi sequence of length $n$. Then the coefficient sequence of each of the polynomials $N(X_n)^{n-s,2s+1}$ and $P(X_n)^{n-s,2s+1}$ is skew-symmetric for each nonnegative integer~$s$, and for real $T>0$ the following hold, as $n\to\infty$:
\begin{enumerate}[(i)]
\setlength{\itemsep}{1ex}
\item \label{itm:cor_1} If $s/n\to T$, then $F(N(X_n)^{n-s,2s+1})\to g(\tfrac14-\tfrac{T}{2},T)$.
\item \label{itm:cor_2} If $s/(2n)\to T$, then $F(P(X_n)^{n-s,2s+1})\to g(\tfrac14-\tfrac{T}{2},T)$.
\end{enumerate}
\end{corollary}
\par
Since the global maximum~$F_a$ of $g(R,T)$ (see~\eqref{eqn:Fa}) occurs when $R=\tfrac14-\tfrac{T}{2}$, Corollary~\ref{cor:main_3} shows that the largest known asymptotic merit factor for a family of binary sequences can be achieved by families of skew-symmetric binary sequences. This is of particular interest in view of Golay's conjecture (see the final paragraph of Section~\ref{sec:intro}).
\par
The rest of the paper is organized as follows. Section~\ref{sec:previous} describes some consequences of our results, including the resolution of several conjectures, the explanation of numerical evidence due to other authors, and the encompassing of numerous special cases that have previously appeared in scattered form. Section~\ref{sec:asymptotic} presents our general method for calculating the asymptotic merit factor of a family of binary sequences and their negaperiodic and periodic versions. Section~\ref{sec:leggal} applies this method to Legendre and Galois sequences to establish Theorems~\ref{thm:main_1} and~\ref{thm:main_2}, respectively, using estimates on character sums. Section~\ref{sec:jacobi} extends the analysis for Legendre sequences to Jacobi sequences and so proves Theorem~\ref{thm:main_3}, using counting results for lattice points in polyhedra. (We have chosen to present the proof of Theorem~\ref{thm:main_1} separately, even though it is a special case of Theorem~\ref{thm:main_3}, in order 
to introduce ideas progressively and maintain clarity of explanation.) Section~\ref{sec:comments} discusses what underlies the negaperiodic and periodic constructions, extends the results of the paper to other binary sequence families, and proposes conjectures for the asymptotic merit factor behavior of two further binary sequence families.

\section{Relationship to Previous Results}
\label{sec:previous}

The results where $T\not=1$ in Theorem~\ref{thm:main_1}~\eqref{itm:12},~\eqref{itm:13}, Theorems~\ref{thm:main_2} and~\ref{thm:main_3}, and Corollary~\ref{cor:main_3} are all new, and prove various conjectures posed in the literature. Theorem~\ref{thm:main_1}~(\ref{itm:12}) shows how $N(X_p)^{r,t}$ can achieve an asymptotic merit factor~$F_a$, as defined in~\eqref{eqn:Fa}, proving a conjecture due to Parker~\cite[Conjecture~4]{Parker2005}, and how $N(X_p)^{0,t}$ can achieve an asymptotic merit factor greater than $6.17$, explaining numerical results presented by Xiong and Hall~\cite[Section~VI]{MR2444577}. Theorem~\ref{thm:main_1}~(\ref{itm:13}) shows how $P(X_p)^{r,t}$ can achieve an asymptotic merit factor~$F_a$, proving a conjecture due to Yu and Gong~\cite[Conjecture~3]{MR2629545}. 
Theorem~\ref{thm:main_2}~(\ref{itm:21}) proves the conjecture of Jedwab and Schmidt \cite[Conjecture 9, Corollary 10]{Jedwab2010} that for all $\theta$ and $r$, the asymptotic merit factor of $Y_{n,\theta}^{r,\floor{nT}}$ is $h(T)$ when $0 < T \le 2$. Theorem~\ref{thm:main_3}~(\ref{itm:31}) shows how $X_n^{r,t}$ can attain an asymptotic merit factor~$F_a$ for composite $n$, explaining numerical evidence reported by Parker~\cite[p.~82]{Parker2005}.
\par
Various special cases of Theorems~\ref{thm:main_1},~\ref{thm:main_2},~\ref{thm:main_3}, and Corollary~\ref{cor:main_3} have appeared in scattered form in the literature. The case $T=1$ of Theorem~\ref{thm:main_1}~(\ref{itm:11}) implies that $X_p^{r,p}$ has asymptotic merit factor $g(R,1)$ if $r/p \to R$ as $p\to\infty$.   Since
\[
\frac{1}{g(R,1)}=\tfrac{1}{6}+8\big(R-\tfrac{1}{4}\big)^2\quad\text{for $0\le R\le \tfrac{1}{2}$},
\]
the maximum asymptotic merit factor that can be attained in this way is~$g(\frac14,1)=6$. This was proved by H{\o}holdt and Jensen~\cite{Hoholdt1988}. Theorem~\ref{thm:main_1}~(\ref{itm:11}) was proved for general $R$ and~$T$ by the present authors~\cite{Jedwab2011}.
\par
The case $T=1$ of Theorem~\ref{thm:main_1}~(\ref{itm:12}) implies that $N(X_p)^{\floor{2pR},2p}$ has asymptotic merit factor $g(R+\frac14,1)$, and so the largest asymptotic merit factor that can be attained in this way is~$6$. Xiong and Hall~\cite[Theorem 3.3]{MR2444577} proved this result for $R=0$.
Schmidt, Jedwab, and Parker~\cite[Theorem 5]{MR2506407} then proved the result for general $R$. The case $T=1$ of Theorem~\ref{thm:main_1}~(\ref{itm:13}) shows that $P(X_p)^{\floor{4pR-p},4p}$ also has asymptotic merit factor $g(R+\tfrac14,1)$, as was proved by Schmidt, Jedwab, and Parker~\cite[Theorem 8]{MR2506407}.
\par
The case $T=1$ of Theorem~\ref{thm:main_2}~(\ref{itm:21}) implies that $Y_{n,\theta}^{r,n}$ has asymptotic merit factor $h(1) = 3$ for all $\theta$ and~$r$. This was proved by Jensen and H{\o}holdt~\cite[Section 5]{MR1008545} (see also Jensen, Jensen, and H{\o}holdt~\cite[Theorem 2.2]{MR1145821}). The case $T=1$ of Theorem~\ref{thm:main_2}~(\ref{itm:22}) and~(\ref{itm:23}) implies a corresponding result for $N(Y_{n,\theta})^{r,2n}$ and $P(Y_{n,\theta})^{r,4n}$, respectively, which was proved by Jedwab and Schmidt~\cite[Theorems 11 and 12]{MR2648554}. Jedwab and Schmidt~\cite[Corollary 7]{Jedwab2010} proved that, for $1 \le T\le 2$ and for all $\theta$, there is a choice of $r$ for each $n$ such that the infimum limit of $F(Y_{n,\theta}^{r,\floor{nT}})$ is at least $h(T)$.
The question as to whether the limit of $F(Y_{n,\theta}^{r,\floor{nT}})$ equals $h(T)$ for all choices of $\theta$ and~$r$ was left as an open problem~\cite[Section 5]{Jedwab2010} and is answered in the affirmative by Theorem~\ref{thm:main_2}~(\ref{itm:21}).
\par
The case $T=1$ of Theorem~\ref{thm:main_3}~(\ref{itm:31}) was proved by Jedwab and Schmidt~\cite[Theorem~2.5]{Jedwab2012} under conditions on the growth rate of $\omega(n)$ that are different from~\eqref{eqn:cond_product}. The case $T=1$ of Theorem~\ref{thm:main_3}~(\ref{itm:32}) was proved by Xiong and Hall~\cite[Theorem~5.2]{MR2444577} for $n=pq$ and $R=0$, where $p$ and $q$ are odd primes satisfying $p\equiv q\equiv 1\pmod 4$, under a more restrictive condition than~\eqref{eqn:cond_product}. The case $T=1$ of Corollary~\ref{cor:main_3} implies that, for $n\equiv 1 \pmod 4$, both $N(X_n)^{0,2n+1}$ and $P(X_n)^{-n,4n+1}$ are skew-symmetric binary sequences, each having asymptotic merit factor~$6$. This was proved by Schmidt, Jedwab, and Parker for prime $n$~\cite[Corollaries~6~and~9]{MR2506407}.

\section{Asymptotic Merit Factor Calculation}
\label{sec:asymptotic}

Let $A$ be a binary sequence of length $n$ with associated polynomial $A(z)$ and write $\e_k=e^{2\pi ik/n}$. 
It turns out that $F(A^{r,t})$, $F(N(A)^{r,t})$, and $F(P(A)^{r,t})$ depend only on the function $L_A$ defined, for $a,b,c\in\Z/n\Z$, by
\[
L_A(a,b,c)=\frac{1}{n^3}\sum_{k\in\Z/n\Z}A(\e_k)A(\e_{k+a})\overline{A(\e_{k+b})}\overline{A(\e_{k+c})}.
\]
In the following two theorems, we shall determine the asymptotic behavior of $F(A^{r,t})$, $F(N(A)^{r,t})$, and $F(P(A)^{r,t})$ when $L_A$ approximates either of the functions $I_n$ and $J_n$ defined, for $a,b,c\in\Z/n\Z$, by
\begin{align*}
I_n(a,b,c)&=\begin{cases}
1 & \text{if one of $a,b,c$ is zero and the other two are equal,}\\
0 & \text{otherwise,}
\end{cases}
\intertext{and}
J_n(a,b,c)&=\begin{cases}
1 & \text{if ($c=a$ and $b=0$) or ($a=b$ and $c=0$),}\\
0 & \text{otherwise.}
\end{cases}
\end{align*}
In Section~\ref{sec:leggal}, we shall establish that the error of this approximation for~$L_A$ vanishes asymptotically for Legendre and Galois sequences, thereby proving Theorems~\ref{thm:main_1} and~\ref{thm:main_2}. 
We shall make repeated use of the elementary counting identities
\begin{align}
\sum_{0\le j,\,j+u<t}1 &= \max(0,t-\abs{u}), 	 \label{eqn:plus_identity} \\
\sum_{0\le j,\,u-j<t}1 &= \max(0,t-\abs{t-1-u}). \label{eqn:minus_identity}
\end{align}
\begin{theorem}
\label{thm:g}
Let $n$ take values in an infinite set of positive integers.  For each $n$, let $V_n$ be a binary sequence of length $n$ and suppose that, as $n\to\infty$,
\begin{equation}
(\log n)^3\max_{a,b,c\,\in\,\Z/n\Z}\,\bigabs{L_{V_n}(a,b,c)-I_n(a,b,c)}\to 0.   \label{eqn:assumption_LI}
\end{equation}
Let $R$ and $T>0$ be real. Then the following hold, as $n\to\infty$:
\begin{enumerate}[(i)]
\setlength{\itemsep}{1ex}
\item If $r/n\to R$ and $t/n\to T$, then $F(V_n^{r,t})\to g(R,T)$.
\item If each $n$ is odd, $r/(2n)\to R$, and $t/(2n)\to T$, then $F(N(V_n)^{r,t})\to g(R+\tfrac{1}{4},T)$.
\item If each $n$ is odd, $r/(4n)\to R$, and $t/(4n)\to T$, then $F(P(V_n)^{r,t})\to g(R,T)$.
\end{enumerate}
\end{theorem}
\begin{proof}
Let $V_n(z)=\sum_{j=0}^{n-1}v_{n,j}\,z^j$ be the polynomial associated with $V_n$ and write $v_{n,j+n}=v_{n,j}$ for all $j$. We treat the three parts of the theorem together by letting the binary sequence $U_n$ be one of $V_n$, $N(V_n)$, or $P(V_n)$. In all three parts, $U_n$ can written in polynomial form as
\[
U_n(z)=\sum_{j=0}^{sn-1}w_jv_{n,j}\,z^j,
\]
where $s\in\{1,4\}$ and $w_j\in\{-1,1\}$ for all $j$. After elementary manipulations, we find from~\eqref{eqn:F_int} that $1+1/F(U_n^{r,t})$ equals
\begin{equation}
\frac{1}{t^2}\sums{0\le j_1,j_2,j_3,j_4<t \\ j_1+j_2=j_3+j_4}(w_{j_1+r}w_{j_2+r}w_{j_3+r}w_{j_4+r})(v_{n,j_1+r}v_{n,j_2+r}v_{n,j_3+r}v_{n,j_4+r}).   \label{eqn:FU}
\end{equation}
Write $\e_k=e^{2\pi ik/n}$. It is readily verified that, for all integers~$j$,
\[
v_{n,j}=\frac{1}{n}\sum_{k\in\Z/n\Z}V_n(\e_k)\,\e_k^{-j}.
\]
A straightforward calculation then shows that, if $j_1,j_2,j_3,j_4$ are integers satisfying $j_1+j_2=j_3+j_4$, then
\begin{equation}
v_{n,j_1}v_{n,j_2}v_{n,j_3}v_{n,j_4}=\frac{1}{n}\sum_{a,b,c\in\Z/n\Z}L_{V_n}(a,b,c)\e_a^{-j_2}\e_b^{j_3}\e_c^{j_4}.   \label{eqn:4v}
\end{equation}
Note that $I_n(a,b,c)$ approximates $L_{V_n}(a,b,c)$ via~\eqref{eqn:assumption_LI}. Consider three cases for the tuple $(a,b,c)\in\Z/n\Z$: (1) $c=a$ and $b=0$, (2) $a=b$ and $c=0$, and (3) $b=c$ and $a=0$. Then $I_n(a,b,c)=1$ if at least one of these conditions is satisfied, and $I_n(a,b,c)=0$ otherwise. The only tuple $(a,b,c)$ that satisfies more than one of these conditions is $(0,0,0)$. We now substitute~\eqref{eqn:4v} into~\eqref{eqn:FU} and reorganize~\eqref{eqn:FU} by writing $L_{V_n}(a,b,c)$ as $I_n(a,b,c)$ plus an error term, and then break the sum involving $I_n(a,b,c)$ into four parts: three sums corresponding to the three cases, and a fourth sum to correct for the triple counting of $(a,b,c)=(0,0,0)$. We keep the sum~$E$ over the error term entire, and thus have
\[
\frac{1}{F(U_n^{r,t})}=-1+A+B+C-2D+E,
\]
where
\begin{align*}
A&=\frac{1}{t^2n}\sums{0 \le j_1,j_2,j_3,j_4 < t \\ j_1+j_2=j_3+j_4} w_{j_1+r}w_{j_2+r}w_{j_3+r}w_{j_4+r}\sum_{a\in\Z/n\Z} \e_a^{j_4-j_2}, \\
B&=\frac{1}{t^2n}\sums{0 \le j_1,j_2,j_3,j_4 < t \\ j_1+j_2=j_3+j_4}
w_{j_1+r}w_{j_2+r}w_{j_3+r}w_{j_4+r}\sum_{b\in\Z/n\Z} \e_b^{j_3-j_2}, \\
C&=\frac{1}{t^2n}\sums{0 \le j_1,j_2,j_3,j_4 < t \\ j_1+j_2=j_3+j_4} w_{j_1+r}w_{j_2+r}w_{j_3+r}w_{j_4+r}\sum_{c\in\Z/n\Z} \e_c^{j_3+j_4+2 r}, \\
D&=\frac{1}{t^2n}\sums{0 \le j_1,j_2,j_3,j_4 < t \\ j_1+j_2=j_3+j_4} w_{j_1+r}w_{j_2+r}w_{j_3+r}w_{j_4+r},\\
E&=\frac{1}{t^2n}\sum_{a,b,c\,\in\,\Z/n\Z}\big[L_{V_n}(a,b,c)-I_n(a,b,c)\big]\,\e_{-a+b+c}^r\\
&\qquad\qquad\times\sums{0 \le j_1,j_2,j_3,j_4 < t \\ j_1+j_2=j_3+j_4}w_{j_1+r}w_{j_2+r}w_{j_3+r}w_{j_4+r}\,\e_a^{-j_2}\e_b^{j_3}\e_c^{j_4}.
\end{align*}
Notice that $A=B$ and there are contributions in $A$ only when $j_4=j_2+mn$ for some $m\in\Z$. When this occurs, we also have $j_1=j_3+ m n$ since $j_1+j_2=j_3+j_4$, so that
\begin{equation}\label{eqn:AplusB}
A+B=\frac{2}{t^2}\sum_{m\in\Z}\enspace\bigg(\sum_{0\le j,\,j+mn<t} w_{j+r}w_{j+r+mn}\bigg)^2.
\end{equation}
Likewise (using $j_4=j_2+m$ instead of $j_4=j_2+m n$), we obtain
\[
D=\frac{1}{t^2n}\sum_{m\in\Z}\enspace\bigg(\sum_{0\le j,\,j+m<t}w_{j+r}w_{j+r+m}\bigg)^2.
\]
Similarly, there are contributions in $C$ only when $j_4=mn-2r-j_3$ for some $m\in\Z$, and therefore
\begin{equation}\label{eqn:C}
C=\frac{1}{t^2}\sum_{m\in\Z}\enspace\bigg(\sum_{0\le j,\,mn-2r-j<t}w_{j+r}w_{mn-(j+r)}\bigg)^2.
\end{equation}
If $t/n$ tends to a positive real number as $n\to\infty$, then assumption~\eqref{eqn:assumption_LI}, combined with Lemma~\ref{lem:sum_abc} (with $v_j=w_{j+r}$) stated below, implies that $E\to0$.  Thus it remains to determine the asymptotic behavior of the sums $A+B$, $C$, and $D$ for the three parts of the theorem.  We shall use the notation $x_n \sim y_n$ to mean that $x_n-y_n\to 0$ as $n\to\infty$. 
\par
{\itshape (i) $U_n=V_n$:} Here we have $s=1$ and $w_j=1$ for all $j$, and we suppose that $r/n\to R$ and $t/n\to T$ as $n\to\infty$. Identities \eqref{eqn:plus_identity} and~\eqref{eqn:minus_identity} give
\begin{align*}
A+B&=\frac{2}{t^2}\sum_{m\in\Z}\max(0,t-\abs{m}n)^2,\\
D  &=\frac{1}{t^2n}\sum_{m\in\Z}\max(0,t-\abs{m})^2,\\
C  &=\frac{1}{t^2}\sum_{m\in\Z}\max(0,t-\abs{t-1-mn+2r})^2,
\end{align*}
and we can then evaluate $D$ exactly as $(2t^2+1)/(3tn)$. Then, since $A+B$ and $C$ are continuous functions of $t$ and $r$, we obtain $-1+A+B+C-2D\to 1/g(R,T)$, as required.
\par
{\itshape (ii) $U_n=N(V_n)$:} Here we have $s=4$ and $w_j=(-1)^{j(j-1)/2}$ for all $j$, and we suppose that each $n$ is odd and $r/(2n)\to R$ and $t/(2n)\to T$ as $n\to\infty$. Since $w_{j+2k}=(-1)^kw_j$ for all $j$, by \eqref{eqn:plus_identity} the contribution to $A+B$ arising by restricting the outer sum in \eqref{eqn:AplusB} to even $m$ is
\[
\frac{2}{t^2}\sum_{m\in\Z}\max(0,t-2\abs{m}n)^2.
\]
Now, for all $j$ and for all odd~$u$ we have $w_j w_{j+u} + w_{j+1} w_{j+1+u}=0$, and therefore if $S$ is a finite set of consecutive integers, we have
\begin{equation}
\Biggabs{\sum_{j\in S}w_jw_{j+u}}\le 1 \quad \text{for odd $u$}.   \label{eqn:sum_wju}
\end{equation}
The terms in the outer sum of $A+B$ are zero whenever $\abs{m}n>t-1$, so that the number of nonzero terms in the outer sum of $A+B$ is bounded by $1+2(t-1)/n$. Using~\eqref{eqn:sum_wju} and the assumption that $n$ is odd, we then find that the contribution to $A+B$ arising by restricting the outer sum to odd~$m$ is at most $2/t^2+4/(tn)$, and therefore
\[
A+B \sim \frac{2}{t^2}\sum_{m\in\Z}\max(0,t-2\abs{m}n)^2.
\]
Likewise,
\[
D \sim \frac{1}{t^2 n}\sum_{m\in\Z}\max(0,t-2\abs{m})^2
\]
and therefore $D \sim t/(3n)$. We proceed similarly to estimate $C$. Here we use that $w_{1-j}=w_j$ for all $j$. It then follows from~\eqref{eqn:sum_wju} that, if $S$ is a finite set of consecutive integers, then
\[
\Biggabs{\sum_{j\in S}w_jw_{u-j}}\le 1 \quad \text{for even $u$}.
\]
We now split the outer sum of $C$ in \eqref{eqn:C} into sums over odd and even $m$, noting that we may neglect contributions arising from the sum over even~$m$ as $n\to\infty$. Since $w_{2k+1-j}=(-1)^kw_j$ for all~$j$, by \eqref{eqn:minus_identity} this gives 
\[
C \sim \frac{1}{t^2}\sum_{m\in\Z}\max(0,t-\abs{t-1-(2m-1)n+2r})^2.
\]
We conclude that $-1+A+B+C-2D\to 1/g(R+\tfrac{1}{4},T)$, as required.
\par
{\itshape (iii) $U_n=P(V_n)$:} Here we have $s=4$ and $w_j=(-1)^{j(j-1)^2/2}$ for all~$j$, and we suppose that each $n$ is odd and $r/(4n)\to R$ and $t/(4n)\to T$ as $n\to\infty$. This can be treated similarly to part~(ii). We have $w_{j+4}=w_j$ and $\sum_{j=0}^3w_jw_{j+u}=0$ for $u\not\equiv 0\pmod 4$, from which we can conclude by \eqref{eqn:plus_identity} that
\[
A+B \sim \frac{2}{t^2}\sum_{m\in\Z}\max(0,t-4\abs{m}n)^2,
\]
and
\[
D \sim \frac{1}{t^2n}\sum_{m\in\Z}\max(0,t-4\abs{m})^2,
\]
so that $D \sim t/(6n)$.
In order to estimate $C$, we use $w_{-j}=w_j$ and~\eqref{eqn:minus_identity} to obtain
\[
C \sim \frac{1}{t^2}\sum_{m\in\Z}\max(0,t-\abs{t-1-4mn+2r})^2.
\]
We conclude that $-1+A+B+C-2D\to 1/g(R,T)$, as required.
\end{proof}
\par
\begin{theorem}
\label{thm:h}
Let $n$ take values in an infinite set of positive integers. For each $n$, let $V_n$ be a binary sequence of length $n$ and suppose that, as $n\to\infty$,
\begin{equation}
(\log n)^3\max_{a,b,c\,\in\,\Z/n\Z}\,\bigabs{L_{V_n}(a,b,c)-J_n(a,b,c)}\to 0.   \label{eqn:cond_J}
\end{equation}
Let $T>0$ be real. Then the following hold, as $n\to\infty$:
\begin{enumerate}[(i)]
\setlength{\itemsep}{1ex}
\item If $t/n\to T$, then $F(V_n^{r,t})\to h(T)$.
\item If each $n$ is odd and $t/(2n)\to T$, then $F(N(V_n)^{r,t})\to h(T)$.
\item If each $n$ is odd and $t/(4n)\to T$, then $F(P(V_n)^{r,t})\to h(T)$.
\end{enumerate}
\end{theorem}
\begin{proof}
The proof of the theorem is similar to the proof of Theorem~\ref{thm:g}, though slightly simpler. Here we consider only two cases for the tuple $(a,b,c)\in\Z/n\Z$: (1) $c=a$ and $b=0$, and (2) $a=b$ and $c=0$, so that $J_n(a,b,c)=1$ if at least one of these conditions is satisfied and $J_n(a,b,c)=0$ otherwise. Letting $U_n$ be one of the sequences $V_n$, $N(V_n)$, or $P(V_n)$, we then have
\[
\frac{1}{F(U_n^{r,t})}=-1+A+B-D+E,
\]
where $A$, $B$, and $D$ are the same expressions (and have the same asymptotic evaluations) as in the proof of Theorem~\ref{thm:g}, but now
\begin{multline*}
E=\frac{1}{t^2n} \sum_{a,b,c\,\in\,\Z/n\Z}\big[L_{V_n}(a,b,c)-J_n(a,b,c)\big]\,\e_{-a+b+c}^r\\
\times\sums{0 \le j_1,j_2,j_3,j_4 < t \\ j_1+j_2=j_3+j_4} w_{j_1+r}w_{j_2+r}w_{j_3+r}w_{j_4+r}\,\e_a^{-j_2}\e_b^{j_3} \e_c^{j_4}.
\end{multline*}
The term $C$ never arises because we have no analogue of case (3) following \eqref{eqn:4v} in the proof of the previous theorem; and we subtract $D$, rather than $2D$ as previously, because the tuple $(a,b,c)=(0,0,0)$ is doubly counted in cases (1) and (2) rather than trebly counted. When $U_n = V_n$, $N(V_n)$, or $P(V_n)$, the proof is completed by observing that, as $n\to\infty$, we have $-1+A+B-D\to 1/h(T)$, and if $t/n$ tends to a positive real number then $E\to 0$ by the assumption~\eqref{eqn:cond_J} and Lemma~\ref{lem:sum_abc}.
\end{proof}
\par
We close this section by proving the result used in the proof of Theorems~\ref{thm:g} and~\ref{thm:h}, which is similar to Lemma~2.2 of~\cite{Jedwab2011} but more widely applicable.
\begin{lemma}
\label{lem:sum_abc}
Let $n$ be a positive integer and write $\e_k=e^{2\pi ik/n}$.  Let $s$ be a positive integer coprime to $n$, and let $v_j\in\C$ be such that $\abs{v_j}\le 1$ and $v_{j+s}=v_j$ for all $j\in\Z$.
Then
\[
\sum_{a,b,c\,\in\Z/n\Z}\Biggabs{\sums{0\le j_1,j_2,j_3,j_4 < t \\ j_1+j_2=j_3+j_4} \!\!\!\!\!\!\!\!\!\! v_{j_1}v_{j_2}v_{j_3}v_{j_4}\,\e_a^{-j_2}\e_b^{j_3}\e_c^{j_4}}\le 936 s^3\max(n,\ceil{t/s})^3(1+\log n)^3.
\]
\end{lemma}
\begin{proof}
Since $\abs{v_j}\le 1$ for all $j$, and the value of $v_j$ depends only on the congruence class of $j$ modulo $s$, the sum to be bounded is at most
\[
\sum_{a,b,c\,\in\Z/n\Z} \,\,\, \sum_{k_2,k_3,k_4=0}^{s-1} \Biggabs{\sums{0\le j_1,j_2,j_3,j_4 < t \\ j_1+j_2=j_3+j_4 \\ (j_2,j_3,j_4)\equiv(k_2,k_3,k_4)\pmod{s}} \!\!\!\!\!\!\!\!\!\! \e_a^{-j_2}\e_b^{j_3}\e_c^{j_4}}.
\]
Reparameterize the inner sum by $(j_1,j_2,j_3,j_4)=(i_1,i_2,i_3,i_4)s+(k_3+k_4-k_2,k_2,k_3,k_4)$ and $(x,y,z)=(-a,b,c)s$. Since $s$ is coprime to~$n$, we obtain
\[
\sum_{k_2,k_3,k_4=0}^{s-1} \,\,\, \sum_{x,y,z \, \in \Z/n\Z} \Biggabs{\sums{(i_1,i_2,i_3,i_4)\in I_1\times I_2\times I_3\times I_4 \\ i_1+i_2=i_3+i_4} \e_x^{i_2} \e_y^{i_3} \e_z^{i_4}},
\]
where each of $I_1$, $I_2$, $I_3$, and $I_4$ is a set of at most $\ceil{t/s}$ consecutive integers (depending on $k_2$, $k_3$, and $k_4$).  Apply Lemma~\ref{lem:ArbitraryIntervals} to the sum over $x,y,z$.
\end{proof}
\par
\begin{lemma}\label{lem:ArbitraryIntervals}
Let $n$ be a positive integer and write $\e_k=e^{2\pi i k/n}$.
Let each of $I_1, I_2, I_3, I_4$ be a finite set of at most $L$ consecutive
integers.  Then
\[
\sum_{a,b,c \, \in \Z/n\Z} \Biggabs{\sums{(i_1,i_2,i_3,i_4)\in I_1 \times I_2 \times I_3 \times I_4\\ i_1+i_2=i_3+i_4}  \e_a^{i_2} \e_b^{i_3} \e_c^{i_4}} \leq 936 \max(n,L)^3 (1+\log n)^3.
\]
\end{lemma}
\begin{proof}
We may assume that each of the sets $I_1,I_2,I_3,I_4$ is nonempty, otherwise the result is trivial. By reparameterizing, we may also assume that $|I_1| \leq |I_2|$ and $|I_3| \leq |I_4|$. Translate $I_1,I_2,I_3$, and $I_4$ to sets $H_1,H_2,H_3$, and $H_4$, respectively, each of whose least element is zero. Then for some $\lambda \in \Z$ the sum to be bounded is
\begin{equation}\label{ShiftedSum}
\sum_{a,b,c \in \Z/n\Z} \Biggabs{\sums{(h_1,h_2,h_3,h_4)\in H_1 \times H_2 \times H_3 \times H_4\\ h_1+h_2=h_3+h_4+\lambda}  \e_a^{h_2} \e_b^{h_3} \e_c^{h_4}}.
\end{equation}
Set $u=2L$. We may assume that $|\lambda| < u$, otherwise the inner sum is empty and the desired bound is immediate.
\par
Let $H_1=\{0,1,\ldots,f\}$ and $H_2=\{0,1,\ldots,g\}$; note that $0 \le f \le g$. Then for a function $S$ of two variables, the sum $\sum_{(h_1,h_2) \in H_1 \times H_2} S(h_1,h_2)$ equals
\[
\sum_{v=0}^{f-1} \sum_{h_1=0}^v S(h_1,v-h_1) + \sum_{v=f}^g \sum_{h_1=0}^f S(h_1,v-h_1) + \sum_{v=g+1}^{f+g} \sum_{h_1=v-g}^f S(h_1,v-h_1).
\]
The range of each of the three inner sums over $h_1$ has the form $jv-w \le h_1 \le kv+x$, where $w\in\{0,|H_2|-1\}$, $x\in\{0,|H_1|-1\}$, and $j,k\in\{0,1\}$.
Apply the same rationale to sums over $(h_3,h_4) \in H_3 \times H_4$ to break the inner sum of \eqref{ShiftedSum} into nine sums (some of which may be empty), each of the form
\[
\sum_{v \in V} \,\,\,\,\, \sum_{h_1=jv-w}^{kv+x} \,\,\,\,\, \sum_{h_3=\ell(v-\lambda)-\beta}^{m(v-\lambda)+\gamma} \e_a^{v-h_1} \e_b^{h_3} \e_c^{v-\lambda-h_3}
\]
where $V$ is a set of consecutive integers in $[0,u)$, the integers $w,x,\beta,\gamma$ satisfy $0 \leq w+x < u$ and $0 \le \beta+\gamma < u$, and $j,k,\ell,m \in \{0,1\}$. By the triangle inequality and some reparameterization, it suffices to show that 
\[
G=\sum_{a,b,c \in \Z/n\Z} \Biggabs{\sum_{v \in V} \sum_{h_1=jv-w}^{kv+x} \,\sum_{h_3=\ell v-y}^{mv+z} \e_a^v \e_b^{h_1} \e_c^{h_3}}
\]
is at most $104 \max(n,L)^3 (1+\log n)^3$, where $V$ is a set of consecutive integers lying in $[0,u)$, the integers $w,x,y,z$ satisfy $0 \le w+x < u$ and $|y+z| < 2u$, and $j,k,\ell,m \in \{0,1\}$.
\par
Now separate $G$ into four sums according to whether each of $b$ and $c$ is 0 to obtain $G=G_1+G_2+G_3+G_4$, where
\begin{align*}
G_1 & = \sums{a,b,c \in \Z/n\Z \\ b,c\not=0} \Biggabs{\sum_{v \in V} \frac{\e_a^v (\e_b^{jv-w}-\e_b^{x+1+k v})(\e_c^{\ell v-y}-\e_c^{z+1+m v})}{(1-\e_b)(1-\e_c)}}, \\
G_2 & = \sums{a,b\in \Z/n\Z \\ b\not=0} \Biggabs{\sum_{v \in V} \big((y+z+1)+(m-\ell)v\big) \frac{\e_a^v(\e_b^{jv-w}-\e_b^{x+1+kv})}{1-\e_b}}, \\
G_3 & = \sums{a,c\in \Z/n\Z \\ c\not=0} \Biggabs{\sum_{v \in V} \big((w+x+1)+(k-j)v\big) \frac{\e_a^v(\e_c^{\ell v-y}-\e_c^{z+1+mv})}{1-\e_c}}, \\
G_4 & = \sum_{a \in \Z/n\Z} \Biggabs{\sum_{v \in V}  \big((w+x+1)+(k-j)v\big) \big((y+z+1)-(m-\ell)v\big) \e_a^v}.
\end{align*}
By the triangle inequality, the constraints $|w+x| < u$ and $|y+z| < 2 u$ and $j,k,\ell,m \in \{0,1\}$, and some reparameterization, we have
\begin{align*}
G_1 & \leq \sums{b,c,d \in \Z/n\Z \\ b,c\not=0} \frac{4}{|1-\e_b|\cdot|1-\e_c|}\biggabs{\sum_{v \in V} \e_d^v}, \\
G_2, G_3 & \leq \sums{b,d\in \Z/n\Z \\ b\not=0} \frac{1}{|1-\e_b|} \left(4u\, \biggabs{\sum_{v \in V} \e_d^v} +2\, \biggabs{\sum_{v \in V} v \e_d^v}\right), \\
G_4 & \leq \sum_{a \in \Z/n\Z} \left(2 u^2 \biggabs{\sum_{v \in V} \e_a^v} + 3 u \biggabs{\sum_{v \in V} v \e_a^v} + \biggabs{\sum_{v \in V} v^2 \e_a^v}\right).
\end{align*}
We next prove by induction on $h \ge 0$ that, for a set $V$ of consecutive integers in $[0,u)$,
\begin{equation}\label{NonzeroContrib}
\sums{a \in \Z/n\Z \\ a\not=0} \Biggabs{\sum_{v \in V} v^h \e_a^v} \leq 2u^h n \log n,
\end{equation}
For the base case $h=0$, we note that $|\sum_{v \in V} \e_a^v| \leq 2 |1-\e_a|^{-1}$ and use the standard bound \cite[p.~136]{MR1790423}
\begin{equation}\label{CrashingBound}
\sum_{a=1}^{n-1} \frac{1}{|1-\e_a|} \leq n \log n.
\end{equation}
For $h > 0$, write $V = \{\sigma,\sigma+1,\dots,\tau-1\}$ and note that
\[
\sum_{v \in V} v^h \e_a^v = \sum_{i=\sigma}^{\tau-2} \,\sum_{v=i+1}^{\tau-1} v^{h-1} \e_a^v + \sigma \sum_{v=\sigma}^{\tau-1} v^{h-1} \e_a^v.
\]
Apply the triangle inequality and the inductive hypothesis to obtain
\[
\sums{a \in \Z/n\Z \\ a\not=0} \Biggabs{\sum_{v \in V} v^h \e_a^v} \leq \big((\tau-\sigma-1)+\sigma\big) 2u^{h-1} n \log n,
\]
which completes the proof of \eqref{NonzeroContrib} since $\tau \le u$. From \eqref{NonzeroContrib}, we find
\[
\sum_{a \in \Z/n\Z} \Biggabs{\sum_{v \in V} v^h \e_a^v} \leq u^{h+1} + 2u^h n \log n,
\]
and we apply this and \eqref{CrashingBound} to the bounds for $G_1$, $G_2$, $G_3$, and $G_4$ to obtain
\begin{align*}
G_1 & \leq 4 (n\log n)^2 (u+2 n\log n) \\
G_2, G_3 & \leq 4 u (n\log n) (u+2 n\log n) + 2 n\log n (u^2 + 2 u n\log n) \\
G_4 & \leq 2 u^2 (u+2 n\log n) + 3 u(u^2 + 2 u n\log n) + (u^3 + 2 u^2 n\log n).
\end{align*}
Since $u=2L$ and $G=G_1+G_2+G_3+G_4$, we conclude that $G \leq 104 \max(n,L)^3 (1+\log n)^3$ as required.
\end{proof}

\section{Legendre and Galois sequences}
\label{sec:leggal}

At the beginning of Section \ref{sec:asymptotic}, it was noted one can compute the merit factor of a binary sequence~$A$ of length~$n$ from the function~$L_A$ defined, for $a,b,c \in \Z/n\Z$, by
\[
L_A(a,b,c)=\frac{1}{n^3}\sum_{k\in\Z/n\Z}A(\e_k)A(\e_{k+a})\overline{A(\e_{k+b})}\overline{A(\e_{k+c})},
\]
where $\e_k=e^{2\pi i k/n}$.
In this section, we combine Theorem~\ref{thm:g} with an estimate of $L_A(a,b,c)$ for Legendre sequences in order to complete the proof of Theorem~\ref{thm:main_1}, and combine Theorem~\ref{thm:h} with an estimate of $L_A(a,b,c)$ for Galois sequences in order to complete the proof of  Theorem~\ref{thm:main_2}.
\par
Theorem~\ref{thm:main_1} is obtained by combining the following lemma with Theorem~\ref{thm:g}, taking $V_n=X_n$ for odd prime $n$.
\begin{lemma}
\label{lem:Omega-Jacobi}
Let $X_p$ be the Legendre sequence of prime length~$p$, as defined in \eqref{LegendreSequenceDefinition}. Then
\[
\max_{a,b,c\,\in\,\Z/p\Z}\,\bigabs{L_{X_p}(a,b,c)-I_p(a,b,c)}\le 18 p^{-1/2}.
\]
\end{lemma}
\begin{proof}
For $\e_k=e^{2\pi i k/p}$, from \eqref{LegendreSequenceDefinition} we have 
\[
X_p(\e_k) -1 = \sum_{j=1}^{p-1} \bigleg{j}{p} \epsilon_k^j,
\]
which is a quadratic Gauss sum and evaluates to $i^{(p-1)^2/4} p^{1/2} \bigleg{k}{p}$ \cite{Gauss}, \cite{MR621882}. It follows from the multiplicativity of the Legendre symbol that
\[
L_{X_p}(a,b,c)=\frac{1}{p} \sum_{x \in \F_p} \bigleg{x(x+a)(x+b)(x+c)}{p} + \Delta,
\]
where $\abs{\Delta}\le 15p^{-1/2}$. The Weil bound~\cite{Weil},~\cite[Theorem~5.41]{MR1429394} shows that the sum over $x$ has magnitude at most~$3p^{1/2}$ when $x(x+a)(x+b)(x+c)$ is not a square in $\F_p[x]$. This polynomial is a square in $\F_p[x]$ if and only if it either has two distinct double roots, in which case the sum over $x$ equals $p-2$, or else has a quadruple root, in which case the sum is $p-1$.
\end{proof}
\par
Theorem~\ref{thm:main_2} is obtained by combining the following lemma with Theorem~\ref{thm:h}, taking $V_n=Y_{n,\theta}$.
\begin{lemma}
\label{lem:indicator-mseq}
Let $Y_{n,\theta}$ be the Galois sequence of length~$n=2^d-1$ with respect to a primitive element $\theta \in \F_{2^d}$, as defined in \eqref{GaloisSequenceDefinition}. Then
\[
\max_{a,b,c\,\in\,\Z/n\Z}\,\bigabs{L_{Y_{n,\theta}}(a,b,c)-J_n(a,b,c)}\le \frac{(n+1)^{3/2}}{n^2}.
\]
\end{lemma}
\begin{proof}
Write $q=2^d=n+1$ and $\e_k=e^{2\pi i k/n}$. Let $\chi\colon \F_q^* \to\C$ be the multiplicative character of order $q-1$ given by $\chi(\theta^j)=\epsilon_j$, so that $\chi^k(\theta^j)=\epsilon_j^k$.
Then from \eqref{GaloisSequenceDefinition},
\[
Y_{n,\theta}(\epsilon_k) = \sum_{x \in \F_q^*} \psi(x) \chi^k(x)
\]
is a Gauss sum.
We use the following facts~\cite[Theorems 5.11 and 5.12]{MR1429394}: (i) $Y_{n,\theta}(1)=-1$; and (ii) $Y_{n,\theta}(\epsilon_k)$ and $Y_{n,\theta}(\epsilon_{-k})$ are complex conjugates, each of magnitude $q^{1/2}$, when $k\not\equiv 0\pmod{n}$.
\par
Now $L_{Y_{n,\theta}}(a,b,c)$ can be written as
\[
\frac{1}{n^3}\sum_{k\in\Z/n\Z}\sum_{w,x,y,z\in\F_q^*}\psi(w+x+y+z)\chi^k(w)\chi^{k+a}(x)\overline{\chi^{k+b}(y) \chi^{k+c}(z)}.
\]
Since $\sum_{k\in\Z/n\Z}\chi^k(v)$ equals $n$ for $v=1$ and equals zero otherwise, we have
\[
L_{Y_{n,\theta}}(a,b,c)=\frac{1}{n^2}\sums{w,x,y,z\in\F_q^*\\wx=yz}\psi(w+x+y+z)\chi^a(x)\overline{\chi^b(y) \chi^c(z)}.
\]
Set $v=w/y=z/x$, and separate out terms with $v=1$ to obtain 
\[
L_{Y_{n,\theta}}(a,b,c)= \delta_b \delta_{a-c} + \frac{1}{n^2}\sums{v,x,y\in\F_q^* \\ v\not=1} \psi((v+1)(x+y)) \chi^{a-c}(x) \chi^{-b}(y) \chi^{-c}(v),
\]
where $\delta_0=1$ and $\delta_u=0$ for nonzero $u$, and we have used the fact that $\sum_{s\in\F_q^*}\chi^u(s)=n \delta_u$ for $u\in\Z/n\Z$.
Reparameterize with $t=(v+1) x$ and $u=(v+1) y$ to get
\begin{align*}
L_{Y_{n,\theta}}(a,b,c)
& = \delta_b \delta_{a-c} + \frac{1}{n^2}\sums{t,u,v\in\F_q^* \\ v\not=1} \!\!\!\!\! \psi(t) \psi(u) \chi^{a-c}(t) \chi^{-b}(u) \chi(v^{-c} (v+1)^{b+c-a}), \\
& = \delta_b \delta_{a-c} + \frac{1}{n^2} Y_{n,\theta}(\epsilon_{a-c}) Y_{n,\theta}(\epsilon_{-b}) \sums{v\in\F_q^*\smallsetminus\{1\}} \chi(v^{-c} (v+1)^{b+c-a}).
\end{align*}
Using facts (i) and (ii), we get the explicit evaluation
\[
L_{Y_{n,\theta}}(a,b,c)=\begin{cases}
1+\frac{n-1}{n^2} & \text{if $a=b=c=0$},\\[1ex]
1-\frac{1}{n^2}   & \text{if $\{0,a\}=\{b,c\}$ and $a\ne 0$},
\end{cases}
\]
which gives the desired result in the case that $J_n(a,b,c)=1$.
\par 
Otherwise we have $\{0,a\}\not=\{b,c\}$ (so that $J_n(a,b,c)=0$).
Then $\delta_b \delta_{a-c}$ vanishes, and the exponents $-c$ and $b+c-a$ in the last sum over $v$ cannot simultaneously vanish. Thus the Weil bound~\cite{Weil}, \cite[Theorem~5.41]{MR1429394} shows that the sum over $v$ has magnitude at most~$q^{1/2}$. This, along with facts (i) and (ii), shows that $|L_{Y_{n,\theta}}(a,b,c)| \leq \frac{(n+1)^{3/2}}{n^2}$.
\end{proof}

\section{Jacobi sequences}
\label{sec:jacobi}

In this section, we prove Theorem~\ref{thm:main_3}. We shall give a detailed proof of part~(i) of Theorem~\ref{thm:main_3}, making use of the machinery developed in the proof of Theorem~\ref{thm:g} together with Lemma~\ref{lem:Omega-Jacobi}. We shall then describe how to modify the proof to establish parts (ii) and~(iii). 
\par
The condition~\eqref{eqn:cond_product} is given, and we suppose that $r/n\to R$ and $t/n\to T$ as $n\to\infty$. Let 
\[
X_n(z)=\sum_{j=0}^{n-1}x_{n,j}\,z^j
\]
be the polynomial associated with the Jacobi sequence of length $n$ and write $x_{n,j+n}=x_{n,j}$ for all $j$. Let $P(n)$ be the set of prime divisors of $n$, so that $n=\prod_{p\in P(n)}p$ since $n$ is square-free. The crucial ingredient of the proof is the representation
\begin{equation}
x_{n,j}=\prod_{p\in P(n)}x_{p,j},
\label{eqn:jacobi_product}
\end{equation}
which is an immediate consequence of the definition of the Jacobi symbol. Then, by the same reasoning as in the beginning of the proof of Theorem~\ref{thm:g}, we find that
\begin{equation}
1+\frac{1}{F(X_n^{r,t})}=\frac{1}{t^2}\sums{0\le j_1,j_2,j_3,j_4<t \\ j_1+j_2=j_3+j_4}\;\prod_{p\in P(n)}\;x_{p,j_1+r}x_{p,j_2+r}x_{p,j_3+r}x_{p,j_4+r}.   \label{eqn:F_prod_elements}
\end{equation}
Also, writing $\z_d=e^{2\pi i/d}$, we see from \eqref{eqn:4v} that, if $j_1,j_2,j_3,j_4$ are integers satisfying $j_1+j_2=j_3+j_4$, then
\[
x_{p,j_1}x_{p,j_2}x_{p,j_3}x_{p,j_4}=\frac{1}{p}\sum_{a,b,c\in\Z/p\Z}L_{X_p}(a,b,c)\;\z_p^{-aj_2}\,\z_p^{bj_3}\,\z_p^{cj_4}.
\]
Substitute into~\eqref{eqn:F_prod_elements} and write $P(n)=\{p_1,p_2,\dots,p_\ell\}$ (where $\ell=\omega(n)$ is the number of prime divisors of $n$) to see that $1+1/F(X_n^{r,t})$ equals
\begin{multline}
\frac{1}{t^2n}\sum_{a_1,b_1,c_1\in\Z/p_1\Z}\cdots\sum_{a_\ell,b_\ell,c_\ell\in\Z/p_\ell\Z} \bigg( \prod_{k=1}^\ell \; L_{X_{p_k}}(a_k,b_k,c_k) \bigg)\\
\times\sums{0\le j_1,j_2,j_3,j_4<t \\ j_1+j_2=j_3+j_4}\;\prod_{k=1}^\ell \; \z_{p_k}^{-a_k(j_2+r)}\,\z_{p_k}^{b_k(j_3+r)}\,\z_{p_k}^{c_k(j_4+r)}.   \label{eqn:F_prod_L}
\end{multline}
For each $p \in P(n)$, write $L_{X_p}(a,b,c)=I_p(a,b,c)+N_p(a,b,c)$. From Lemma~\ref{lem:Omega-Jacobi}, we have
\[
\max_{a,b,c\in\Z/p\Z}\;\abs{N_p(a,b,c)}\le 18p^{-1/2}\le 18\kappa(n)^{-1/2}
\]
(where $\kappa(n)$ is the smallest prime divisor of $n$). Henceforth, let $n\ge n_0$, where $n_0$ is the smallest $n$ such that $18\kappa(n)^{-1/2} \le 1$ for all $n \ge n_0$. Such an $n_0$ exists since $\kappa(n)\to\infty$, by~\eqref{eqn:cond_product}. Then, expanding the first product in~\eqref{eqn:F_prod_L} into $2^\ell$ terms, all but one of which contains at least one factor $N_{p_k}(a_k,b_k,c_k)$, we see that $1+1/F(X_n^{r,t})$ equals
\begin{equation}
\frac{1}{t^2 n}\sums{0\le j_1,j_2,j_3,j_4<t \\ j_1+j_2=j_3+j_4}\;\prod_{p\in P(n)}\;\sum_{a,b,c\in\Z/p\Z}I_p(a,b,c)\;\z_p^{-a(j_2+r)}\,\z_p^{b(j_3+r)}\,\z_p^{c(j_4+r)},   \label{eqn:F_product_I}
\end{equation}
plus an error term whose magnitude is bounded by
\[
\frac{18 (2^{\ell}-1)}{t^2n\,\kappa(n)^{1/2}}\!\!\!\!\sum_{a_1,b_1,c_1\in\Z/p_1\Z}\!\cdots\!\sum_{a_\ell,b_\ell,c_\ell\in\Z/p_\ell\Z}\Biggabs{\sums{0\le j_1,j_2,j_3,j_4<t \\ j_1+j_2=j_3+j_4}\;\prod_{k=1}^\ell \z_{p_k}^{-a_kj_2}\,\z_{p_k}^{b_kj_3}\,\z_{p_k}^{c_kj_4}}.
\]
By the Chinese Remainder Theorem, and replacing $2^\ell$ by $2^{\omega(n)}$, this error term equals
\begin{equation}
\frac{18(2^{\omega(n)}-1)}{t^2n\,\kappa(n)^{1/2}}\sum_{a,b,c\in\Z/n\Z} \;\Biggabs{\sums{0\le j_1,j_2,j_3,j_4<t \\ j_1+j_2=j_3+j_4}\; \z_n^{-aj_2}\,\z_n^{bj_3}\,\z_n^{cj_4}}.   \label{eqn:product-error-term-bound}
\end{equation}
Now we turn back to the main term~\eqref{eqn:F_product_I}. Proceeding with three cases for $(a,b,c)$, as in the proof of Theorem~\ref{thm:g}, we find that, for integral $j$, $k$, and~$\ell$,
\[
\frac{1}{p}\sum_{a,b,c\in\Z/p\Z}\!\!\!\!I_p(a,b,c)\;\z_p^{-aj}\,\z_p^{bk}\,\z_p^{c\ell}=\delta_p(\ell-j)+\delta_p(k-j)+\delta_p(k+\ell)-\frac{2}{p},
\]
where, for integral $m$ and $j$,
\[
\delta_m(j)=\begin{cases}
1 & \text{if $m\mid j$}\\
0 & \text{otherwise}.
\end{cases}
\]
Hence,~\eqref{eqn:F_product_I} equals
\[
\frac{1}{t^2}\sums{0\le j_1,j_2,j_3,j_4<t \\ j_1+j_2=j_3+j_4}\;\prod_{p\in P(n)}\;\left(\delta_p(j_4-j_2)+\delta_p(j_3-j_2)+\delta_p(j_3+j_4+2r)-\frac{2}{p}\right).
\]
By expanding the product, this expression can be written as
\[
\sum_{[P_0:P_1:P_2:P_3]=P(n)}\frac{(-2)^{\abs{P_0}}}{t^2\,P_0^\times}\sums{0\le j_1,j_2,j_3,j_4<t \\ j_1+j_2=j_3+j_4}\!\!\delta_{P_1^\times}(j_4-j_2)\delta_{P_2^\times}(j_3-j_2)\delta_{P_3^\times}(j_3+j_4+2r),
\]
where we write the sum over $[P_0:P_1:P_2:P_3]=P(n)$ to mean the sum over all ordered partitions of $P(n)$ into sets $P_0,P_1,P_2,P_3$, and where we write $P_k^\times$ to mean $\prod_{p\in P_k}p$. We partition this sum by separating the three summands where $P_1$, $P_2$, or $P_3$ equals $P(n)$ and so have
\[
\frac{1}{F(X_n^{r,t})}=-1+A+B+C+D+E,
\]
where
\begin{align*}
A&=\frac{1}{t^2}\sums{0\le j_1,j_2,j_3,j_4<t \\ j_1+j_2=j_3+j_4}\delta_n(j_4-j_2),\\
B&=\frac{1}{t^2}\sums{0\le j_1,j_2,j_3,j_4<t \\ j_1+j_2=j_3+j_4}\delta_n(j_3-j_2),\\
C&=\frac{1}{t^2}\sums{0\le j_1,j_2,j_3,j_4<t \\ j_1+j_2=j_3+j_4}\delta_n(j_3+j_4+2r),\\
D&=\sums{[P_0:P_1:P_2:P_3]=P(n) \\ P_1,P_2,P_3\ne P(n)}\frac{(-2)^{\abs{P_0}}}{t^2P_0^\times}\\
&\qquad\qquad\times\sums{0\le j_1,j_2,j_3,j_4<t \\ j_1+j_2=j_3+j_4}\delta_{P_1^\times}(j_4-j_2)\delta_{P_2^\times}(j_3-j_2)\delta_{P_3^\times}(j_3+j_4+2r),
\end{align*}
and $E$ is an error term whose magnitude is bounded by~\eqref{eqn:product-error-term-bound}. The sums $A$, $B$, and $C$ are identical to those in the proof of Theorem~\ref{thm:g} (i), and $E\to0$ by Lemma~\ref{lem:sum_abc} and~\eqref{eqn:cond_product}, because $t/n$ tends to a positive real number. We now show that $D\to-4T/3$, and therefore $-1+A+B+C+D\to 1/g(R,T)$, which completes the proof of part~(i).
\par
Lemma~\ref{lem:lattice_count}~(i) (to be proved below) shows that the inner sum of $D$ equals
\[
\frac{2t^3}{3P_1^\times P_2^\times P_3^\times},
\]
plus an error term whose magnitude is at most 
\[
\frac{4572\,\max(t,P_1^\times,P_2^\times,P_3^\times)^2 \max(P_1^\times,P_2^\times,P_3^\times)}{P_1^\times P_2^\times P_3^\times}.
\]
All partitions involved in the outer sum of $D$ satisfy $\max(P_1^\times,P_2^\times,P_3^\times)\le n/\kappa(n)$, because none of $P_1$, $P_2$, and $P_3$ equals $P(n)$. We further assume that $n\ge n_1$, where $n_1$ is the smallest $n$ such that $n/\kappa(n)\le t$ for all $n\ge n_1$. Such an $n_1$ exists since $t/n$ tends to a positive real number and $\kappa(n)\to\infty$ as $n\to\infty$ by~\eqref{eqn:cond_product}. Therefore $\max(t,P_1^\times,P_2^\times,P_3^\times) = t$, and the error term for the inner sum of $D$ has magnitude at most
\[
\frac{4572\,t^2n}{P_1^\times P_2^\times P_3^\times \kappa(n)}.
\]
Therefore each summand of the outer sum of $D$ equals
\[
\frac{2t}{3n}\,(-2)^{\abs{P_0}},
\]
plus an error term whose magnitude is at most
\begin{equation}
\frac{4572\cdot2^{\abs{P_0}}}{\kappa(n)}.   \label{eqn:error-term}
\end{equation}
Hence $D$ equals
\[
\frac{2t}{3n}\Bigg(\sum_{[P_0:P_1:P_2:P_3]=P(n)}(-2)^{\abs{P_0}}-3\Bigg),
\]
plus $4^{\omega(n)}-3$ error terms each with magnitude at most~\eqref{eqn:error-term}. The principal term for $D$ then evaluates to
\[
\frac{2t}{3n}\Bigg(\sum_{j=0}^{\omega(n)}{\omega(n)\choose j}3^j(-2)^{\omega(n)-j}-3\Bigg)=-\frac{4t}{3n},
\]
which tends to $-4T/3$,
while the sum over the $4^{\omega(n)}-3$ error terms has magnitude smaller than
\[
\frac{4572}{\kappa(n)}\sum_{j=0}^{\omega(n)}{\omega(n)\choose j}3^j\,2^{\omega(n)-j}=\frac{4572\cdot 5^{\omega(n)}}{\kappa(n)},
\]
which by~\eqref{eqn:cond_product} tends to zero as $n\to\infty$.  Therefore $D\to -4T/3$, as required.
\par
We now sketch how to prove parts (ii) and~(iii). We treat both cases together by letting $U_n$ be either $N(X_n)$ or $P(X_n)$. The condition~\eqref{eqn:cond_product} is given; for part (ii) we suppose that $r/(2n)\to R$ and $t/(2n)\to T$ as $n\to\infty$, and for part (iii) we suppose that $r/(4n)\to R$ and $t/(4n)\to T$ as $n \to \infty$.
In polynomial form, we have
\[
U_n(z)=\sum_{j=0}^{4n-1}w_j\,\bigg(\prod_{p\in P(n)}\;x_{p,j}\bigg)z^j,
\]
where $w_j=(-1)^{j(j-1)/2}$ for $U_n=N(X_n)$ and $w_j=(-1)^{j(j-1)^2/2}$ for $U_n=P(X_n)$. Then, proceeding as in the proof of part~(i), we arrive at
\[
\frac{1}{F(U_n^{r,t})}=-1+A+B+C+D+E,
\]
where
\begin{align*}
A&=\frac{1}{t^2}\sums{0\le j_1,j_2,j_3,j_4<t \\ j_1+j_2=j_3+j_4}w_{j_1+r}w_{j_2+r}w_{j_3+r}w_{j_4+r}\;\delta_n(j_4-j_2),\\
B&=\frac{1}{t^2}\sums{0\le j_1,j_2,j_3,j_4<t \\ j_1+j_2=j_3+j_4}w_{j_1+r}w_{j_2+r}w_{j_3+r}w_{j_4+r}\;\delta_n(j_3-j_2),\\
C&=\frac{1}{t^2}\sums{0\le j_1,j_2,j_3,j_4<t \\ j_1+j_2=j_3+j_4}w_{j_1+r}w_{j_2+r}w_{j_3+r}w_{j_4+r}\;\delta_n(j_3+j_4+2r),\\
D&=\sums{[P_0:P_1:P_2:P_3]=P(n) \\ P_1,P_2,P_3\ne P(n)}\frac{(-2)^{\abs{P_0}}}{t^2P_0^\times}\sums{0\le j_1,j_2,j_3,j_4<t \\ j_1+j_2=j_3+j_4}w_{j_1+r}w_{j_2+r}w_{j_3+r}w_{j_4+r}\\[1.5ex]
&\qquad\qquad\qquad\qquad\times\delta_{P_1^\times}(j_4-j_2)\delta_{P_2^\times}(j_3-j_2)\delta_{P_3^\times}(j_3+j_4+2r),  
\end{align*}
and $E$ is an error term whose magnitude is, for all sufficiently large $n$, bounded by
\[
\frac{18(2^{\omega(n)}-1)}{t^2n\,\kappa(n)^{1/2}}\sum_{a,b,c\in\Z/n\Z} \;\Biggabs{\sums{0\le j_1,j_2,j_3,j_4<t \\ j_1+j_2=j_3+j_4}\!\!\!\!w_{j_1+r}w_{j_2+r}w_{j_3+r}w_{j_4+r}\, \z_n^{-aj_2}\,\z_n^{bj_3}\,\z_n^{cj_4}}.
\]
The sums $A$, $B$, and $C$ are the same as in the corresponding parts of the proof of Theorem~\ref{thm:g}, and $E\to0$ by Lemma~\ref{lem:sum_abc} and \eqref{eqn:cond_product} because $t/n$ tends to a positive real number. By invoking Lemma~\ref{lem:lattice_count} (ii) and (iii), we can show, by proceeding as in the proof of part~(i), that $D \sim -2t/(3n)$ for $U_n=N(X_n)$ and $D \sim -t/(3n)$ for $U_n=P(X_n)$, from which parts (ii) and (iii) follow.
\qed\smallskip
\par
To prove Lemma~\ref{lem:lattice_count}, which was invoked in the proof of Theorem~\ref{thm:main_3}, we require the following lemma.
\begin{lemma}
\label{lem:geometric}
Let $t$ be a nonnegative real number and define the half-open polyhedron
\[
C=\big\{(x,y,z)\in\R^3:0\le x,y,z,y+z-x<t\big\}.
\]
Let $a$, $b$, and $c$ be positive integers of the same parity. Define the lattice
\[
\Lambda=\big\{(x,y,z)\in\Z^3:x\equiv y\!\!\!\!\pmod a,\,x\equiv z\!\!\!\!\pmod b,\,y\equiv -z\!\!\!\!\pmod c\big\}
\]
and let $K$ be a translation of $\Lambda$. Then
\[
\biggabs{\abs{K\cap C}-\frac{2t^3}{3abc}}\le\frac{4572\max(t,a,b,c)^2\max(a,b,c)}{abc}
\]
if $a$, $b$, and $c$ are odd, and 
\[
\biggabs{\abs{K\cap C}-\frac{4 t^3}{3abc}}\le\frac{1332\max(t,a,b,c)^2\max(a,b,c)}{abc}
\]
if $a$, $b$, and $c$ are even.
\end{lemma}
\begin{proof}
A standard calculation shows that the volume of $C$ is $\vol(C)=2t^3/3$. For positive real $d$, let $C_d^-$ be the set of points within $C$ that are at distance more than $d$ from the boundary of $C$, and let $C_d^+$ be the set of points lying within $C$ or no further than distance $d$ from some point in $C$. Then $C_d^-\subseteq C\subseteq C_d^+$, and by translating the planes bounding $C$ inward or outward, it can be shown that
\begin{equation}
\vol(C_d^-)\ge\tfrac{2}{3}\big(t-2\sqrt{3}d\big)^3\quad\text{and}\quad\vol(C_d^+)\le\tfrac{2}{3}\big(t+2\sqrt{3}d\big)^3.   \label{eqn:vol_minus_plus}
\end{equation}
\par
Let $v$ and $\ell$ be the volume and the largest diagonal of the fundamental parallelepiped of $\Lambda$, respectively. Then $\abs{K\cap C}$ is at least the number of parallelepipeds of $K$ wholly contained in $C$, which is at least the number intersecting $C_\ell^-$, so that $\abs{K\cap C}$ is at least $\vol(C_\ell^-)/v$. Likewise, $\abs{K\cap C}$ is at most the number of parallelepipeds of $K$ intersecting $C$, which is at most the number wholly contained in $C_\ell^+$, and so $\abs{K\cap C}$ is at most $\vol(C_\ell^+)/v$.
\par
Now, if $a$, $b$, and $c$ are odd, it is readily verified that $\Lambda$ is generated by 
\[
\tfrac{1}{2}(c+a,c-a,c+a),\quad\tfrac{1}{2}(c+b,c+b,c-b),\quad(c,c,c),
\]
from which we find that $v=abc$ and (by the triangle inequality) $\ell\le 3\sqrt{3}\max(a,b,c)$, and the result follows from \eqref{eqn:vol_minus_plus}. On the other hand, if $a$, $b$, and $c$ are even,  $\Lambda$ is generated by 
\[
\tfrac{1}{2}(a,-a,a),\quad\tfrac{1}{2}(b,b,-b),\quad\tfrac{1}{2}(c,c,c),
\]
and $v=abc/2$ and $\ell\le 3\sqrt{3}\max(a,b,c)/2$.
\end{proof}
\par
We now prove the lemma that was invoked in the proof of Theorem~\ref{thm:main_3}.
\begin{lemma}
\label{lem:lattice_count}
Let $r$ be an integer, let $t$ be a nonnegative integer, and let $a$, $b$, and $c$ be odd positive integers. For some $w_j$ with $j\in\Z$, consider the sum
\begin{equation}
\sums{0\le j_1,j_2,j_3,j_4<t \\ j_1+j_2=j_3+j_4}w_{j_1+r}w_{j_2+r}w_{j_3+r}w_{j_4+r}\delta_a(j_4-j_2)\delta_b(j_3-j_2)\delta_c(j_3+j_4+2r),   \label{eqn:delta_sum}
\end{equation}
where $\delta_m(j)$ equals $1$ if $m\mid j$ and equals $0$ otherwise.
\begin{enumerate}[(i)]
\item Let $S_1(a,b,c)$ be the sum~\eqref{eqn:delta_sum}, where $w_j=1$ for all $j\in\Z$. Then
\[
\biggabs{S_1(a,b,c)-\frac{2t^3}{3abc}}\le\frac{4572\max(t,a,b,c)^2\max(a,b,c)}{abc}.
\]
\item Let $S_2(a,b,c)$ be the sum~\eqref{eqn:delta_sum}, where $w_j=(-1)^{j(j-1)/2}$ for all $j\in\Z$. Then
\[
\biggabs{S_2(a,b,c)-\frac{t^3}{3abc}}\le\frac{42624\max(t,a,b,c)^2\max(a,b,c)}{abc}.
\]
\item Let $S_3(a,b,c)$ be the sum~\eqref{eqn:delta_sum}, where $w_j=(-1)^{j(j-1)^2/2}$ for all $j\in\Z$. Then 
\[
\biggabs{S_3(a,b,c)-\frac{t^3}{6abc}}\le\frac{42624\max(t,a,b,c)^2\max(a,b,c)}{abc}.
\]
\end{enumerate}
\end{lemma}
\begin{proof}
For part (i), let $C$ and $\Lambda$ be as in Lemma~\ref{lem:geometric} and let $K=\Lambda-(r,r,r)$. Then
\[
S_1(a,b,c)=\abs{K\cap C},
\]
and (i) follows from Lemma~\ref{lem:geometric} since $a$, $b$, and $c$ have the same parity.
\par
For parts (ii) and (iii), we claim that when $h_1+h_2=h_3+h_4$, the value of $w_{h_1} w_{h_2} w_{h_3} w_{h_4}$ depends only on the congruence class modulo $4$ of $h_4-h_2$, $h_3-h_2$, and $h_3+h_4$. Indeed, for part (ii) we have 
\[
w_{h_1} w_{h_2} w_{h_3} w_{h_4}=(-1)^{(h_4-h_2)(h_3-h_2)}
\]
whenever $h_1+h_2=h_3+h_4$, while for part (iii) we have 
\[
w_{h_1} w_{h_2} w_{h_3} w_{h_4}=
\begin{cases}
(-1)^{(h_4-h_2)(h_3-h_2)/2} & \text{if $(h_4-h_2)(h_3-h_2)$ is even},\\
(-1)^{(h_3+h_4)/2}        & \text{otherwise}
\end{cases}
\]
whenever $h_1+h_2=h_3+h_4$. For either part, define $\sigma\colon \Z^3\to\{-1,1\}$ so that $w_{h_1} w_{h_2} w_{h_3} w_{h_4}=\sigma(h_4-h_2,h_3-h_2,h_3+h_4)$ whenever $h_1+h_2=h_3+h_4$, and reparameterize~\eqref{eqn:delta_sum} to obtain
\begin{equation}\label{eqn:delta_sum_nporp}
\sums{0 \leq k,\ell,m < 4 \\ m\equiv k+\ell \!\!\!\! \pmod{2}} \!\!\!\!\! \sigma(k,\ell,m) \!\!\!\!\! \sums{0\le j_1,j_2,j_3,j_4<t \\ j_1+j_2=j_3+j_4 \\ j_4-j_2 \equiv k \!\!\!\! \pmod{4} \\ j_3-j_2 \equiv \ell \!\!\!\! \pmod{4} \\ j_3+j_4+2r\equiv m \!\!\!\! \pmod{4}} \!\!\!\!\!\!\!\!\!\! \delta_a(j_4-j_2)\delta_b(j_3-j_2)\delta_c(j_3+j_4+2r).
\end{equation}
Since $a$, $b$, and $c$ are odd, by the Chinese Remainder Theorem each of the $32$ inner sums counts the number of points of some translate of the lattice
\[
\Lambda=\big\{(x,y,z)\in\Z^3:x\equiv y\!\!\!\!\pmod{4 a},\,x\equiv z\!\!\!\!\pmod{4 b},\,y\equiv -z\!\!\!\!\pmod{4 c}\big\}
\]
lying within the half-open polyhedron $C$ defined in Lemma~\ref{lem:geometric}. By Lemma \ref{lem:geometric}, each of these $32$ inner sums equals $t^3/(48abc)$ plus an error term of magnitude at most
\begin{equation}
\frac{1332\max(t,a,b,c)^2 \max(a,b,c)}{abc}.   \label{eqn:lattice_count_error_term}
\end{equation}
In part~(ii), $\sigma(k,\ell,m)$ equals $+1$ for $24$ of the triples $(k,\ell,m)$ in the summation and equals $-1$ for the remaining $8$ triples, so~\eqref{eqn:delta_sum_nporp} equals $t^3/(3abc)$ plus an error term whose magnitude is at most $32$ times~\eqref{eqn:lattice_count_error_term}. In part~(iii), $\sigma(k,\ell,m)$ equals $+1$ for $20$ of the triples $(k,\ell,m)$ in the summation and equals $-1$ for the remaining $12$ triples, so~\eqref{eqn:delta_sum_nporp} equals $t^3/(6 abc)$ plus an error term whose magnitude is at most $32$ times~\eqref{eqn:lattice_count_error_term}.
\end{proof}

\section{Closing Comments}
\label{sec:comments}

We close with a discussion of what underlies the negaperiodic and periodic constructions, some generalizations of our results to other binary sequence families involving combinations of Legendre and Galois sequences, and some conjectures on the asymptotic merit factor behavior of two binary sequence families examined by other authors. We hope this will stimulate further research.

\subsection{What underlies the negaperiodic and periodic constructions?}
\label{sec:comments_under}

Let $V=(v_0,v_1,\dots,v_{n-1})$ and $W=(w_0,w_1,\dots,w_{s-1})$ be binary sequences of length~$n$ and~$s$, respectively, and write $v_{j+n}=v_j$ and $w_{j+s}=w_j$ for all $j\in\Z$. Define the \emph{product sequence} formed from $V$ and $W$ to be the length~$ns$ coefficient sequence of 
\[
(V\otimes W)(z)=\sum_{j=0}^{ns-1}v_jw_j\,z^j.
\]
Then we can write $V = V \otimes (+)$ and $N(V)=V\otimes (+,+,-,-)$ and $P(V)=V\otimes (+,+,-,+)$, and it is natural to ask whether the methods of this paper can be applied to $V\otimes W$ when $W$ is not one of $(+)$, $(+,+,-,-)$, and $(+,+,-,+)$. 
\par
Indeed, it is readily shown that the same method used to prove Theorem~\ref{thm:h}~(ii) for $N(V)$ can be applied to $V \otimes W$ for general $W$, under the sufficient conditions that $s$ is even, $\gcd(n,s)=1$, and
\begin{equation}
\sum_{j=0}^{s-1}w_jw_{j+u}=\begin{cases}
s  & \text{for $u\equiv 0\pmod s$},\\
-s & \text{for $u\equiv s/2\pmod s$},\\
0  & \text{otherwise}.
\end{cases}
\label{eqn:negaperiodic_perfect}
\end{equation}
The sequence $(+,+,-,-)$ satisfies these conditions, and gives rise to the negaperiodic construction $N(V) = V \otimes (+,+,-,-)$. The sequence $(+,-)$ also satisfies these conditions, but the resulting product sequence $V \otimes (+,-)$ trivially has the same merit factor properties as~$V$.\footnote{Let $U=V\otimes (+,-)$. Then $U^{r,t}$ arises by negating every other element of $V^{r,t}$, so that the aperiodic autocorrelation of $U^{r,t}$ is obtained from that of $V^{r,t}$ by negating the values at odd shifts, thus preserving the merit factor.} Since the existence of a binary sequence satisfying~\eqref{eqn:negaperiodic_perfect} for even $s>2$ is equivalent to the existence of a $(s/2, 2, s/2, s/4)$ relative difference set $R$ in $\Z/s\Z$ (via the correspondence $j \in R$ if and only if $w_j=-1$), standard nonexistence results for relative difference sets in cyclic groups show that there are no such binary sequences for even $s>4$ \cite[Result~4.8]{MR1157478}, \cite[Corollary~6]{MR2413985}; see \cite[
Appendix~VI]{MR2045804} for a direct proof. Therefore there are no binary sequences~$W$ satisfying the sufficient conditions for $s > 4$.
\par
Likewise, the same method used to prove Theorem~\ref{thm:g}~(ii) for $N(V)$ can be applied to $V \otimes W$ for general $W$, under the same sufficient conditions as above together with the additional condition 
\begin{equation}
w_{k-j}=w_j \quad \mbox{for all $j\in\Z$ and some integer $k$}.
\label{eqn:reversible}
\end{equation}
This enlarged set of conditions is satisfied by all the sequences that satisfy the original set of conditions, namely the sequences $(+,+,-,-)$, $(+,-)$, and their cyclic shifts.
\par
The same method used to prove Theorem~\ref{thm:h}~(iii) for $P(V)$ can be applied to $V \otimes W$ for general $W$, under the sufficient conditions that $\gcd(n,s)=1$ and
\begin{equation}
\sum_{j=0}^{s-1}w_jw_{j+u}=\begin{cases}
s & \text{for $u\equiv 0\pmod s$},\\
0 & \text{otherwise}.
\end{cases}
\label{eqn:periodic_perfect}
\end{equation}
The sequences $(+,+,-,+)$ and $(+)$ satisfy these conditions, and give rise to the periodic construction $P(V) = V \otimes (+,+,-,+)$ and the trivial construction $V = V\otimes (+)$, respectively. The existence of a binary sequence satisfying~\eqref{eqn:periodic_perfect} for $s>1$ is equivalent to the existence of an $(s, (s-\sqrt{s})/2, (s-2\sqrt{s})/4)$-difference set in $\Z/s\Z$, and there are no such binary sequences for $4 < s < 4 \cdot 11715^2$ \cite[Corollary~4.5]{MR2211106}.
\par
Likewise, the same method used to prove Theorem~\ref{thm:g}~(iii) for $P(V)$ can be applied to $V \otimes W$ for general $W$, under the same sufficient conditions from the previous paragraph together with the additional condition~\eqref{eqn:reversible}. This additional condition constrains the difference set to have multiplier~$-1$, and a classical nonexistence result on difference set multipliers shows that there are no such sequences for $s>4$~\cite[Corollary~3.7]{MR0166115}. 

\subsection{Product of Legendre and Galois sequences}

Using the operator~$\otimes$ defined in Section~\ref{sec:comments_under}, we consider product sequences involving Legendre and Galois sequences. As previously, we write $X_p$ for the Legendre sequence of length~$p$, and $Y_{n,\theta}$ for the Galois sequence of length $n=2^d-1$ with respect to a primitive $\theta \in \F_{2^d}$. 
\par
Let $P$ be a set of odd primes, and let $M$ be a set of Mersenne numbers (having the form $2^d-1$ for integral $d$) such that $P$ and $M$ are disjoint and the elements of $P\cup M$ are pairwise coprime. For each $2^d-1 \in M$, choose a primitive element $\theta\in\F_{2^d}$ and consider the product sequence
\begin{equation}
\Big(\bigotimes_{p\in P}X_p\Big) \otimes \Big(\bigotimes_{n\in M}Y_{n,\theta}\Big)
\label{eqn:product_construction}
\end{equation}
of length $(\prod_{p\in P}p)(\prod_{n\in M}n)$.
If $M$ is empty, then by~\eqref{eqn:jacobi_product} the product sequence \eqref{eqn:product_construction} is a Jacobi sequence and its asymptotic merit factor behavior is the same as that of a Legendre sequence (see Theorem~\ref{thm:main_3}). Otherwise, the product sequence involves at least one Galois sequence. In that case, a straightforward (albeit notationally cumbersome) generalization of the proof of Theorem~\ref{thm:main_3} shows that, under suitable conditions on the growth rate of $\abs{P\cup M}$ and $\min (P\cup M)$, the asymptotic merit factor behavior of the product sequence~\eqref{eqn:product_construction} and its negaperiodic and periodic versions is the same as that of a Galois sequence (see Theorem~\ref{thm:main_2}).

\subsection{Gordon-Mills-Welch sequences and Sidelnikov sequences}

Let $F=\F_{2^d}$ be the finite field with $2^d$ elements and let $K$ be a subfield of $F$ of size $2^k$ (so that $k$ divides $d$). The \emph{relative trace} $\Tr_{F/K}:F\to K$ is given by
\[
\Tr_{F/K}(y)=\sum_{j=0}^{d/k-1}y^{2^{jk}}.
\]
Let $\psi$ be the canonical additive character of $K$, let $\theta$ be a primitive element of $F$, and let $\ell$ be an integer coprime to $2^k-1$. The coefficient sequence of the polynomial
\[
\sum_{j=0}^{n-1}\psi\big(\Tr_{F/K}(\theta^j)^\ell\big)z^j
\]
is called a \emph{Gordon-Mills-Welch sequence} of length~$n=2^d-1$~\cite{MR751329} with respect to $\theta$, $k$, $\ell$. The special case $\ell=1$ reduces to a Galois sequence. In 1991, Jensen, Jensen and H{\o}holdt asked how the asymptotic merit factor of a Gordon-Mills-Welch sequence behaves~\cite{MR1145821}. Based on numerical evidence, we conjecture that the generalization from a Galois sequence to a Gordon-Mills-Welch sequence does not affect the asymptotic merit factor, and that the same holds for the negaperiodic and periodic versions of these sequences. 
\begin{conjecture}
\label{con:GMW}
For each $n=2^d-1$, choose a primitive $\theta \in \F_{2^d}$, and $k$ dividing $d$, and $\ell$ coprime to $2^k-1$. Then the asymptotic merit factor of the Gordon-Mills-Welch sequence of length~$n$ (and its negaperiodic and periodic versions) with respect to $\theta$, $k$, $\ell$ is the same as that of a Galois sequence as specified in Theorem~\ref{thm:main_2}.
\end{conjecture}
\par
Now let $q$ be an odd prime power, and let $\theta$ be a primitive element of $\F_q$. Let $\eta:\F_q\to\{1,-1\}$ be the quadratic character on the nonzero elements of $\F_q$, and extend $\eta$ (in a nonstandard way) via $\eta(0)=1$. The coefficient sequence of the polynomial
\[
Z_{n,\theta}(z) = \sum_{j=0}^{q-2}\eta(\theta^j+1)z^j
\]
is called a \emph{Sidelnikov sequence} of length $q-1$ with respect to $\theta$~\cite{MR0305913}. Based on numerical evidence, we conjecture that the asymptotic merit factor of a Sidelnikov sequence is the same as that of a Galois sequence as specified in Theorem~\ref{thm:main_2}~(i).\footnote{Huo~\cite{huo-masters} presents numerical evidence suggesting that the merit factor of the nonbinary analogues of the Sidelnikov sequences (which use multiplicative characters of higher order in place of the quadratic character) might also have the same asymptotic behavior. In their paper on Fekete-like polynomials, Hare and Yazdani~\cite{MR2660887} present numerical evidence suggesting that a particular cyclic shift of a Sidelnikov sequence (namely the coefficient sequence of $\sum_{j=0}^{q-2}\eta(\theta^j-1)z^j$) has asymptotic merit factor $3$.} (Since the length of a Sidelnikov sequence is even, there is no negaperiodic or periodic version to consider.) 
\begin{conjecture}
\label{con:Sidelnikov}
For each odd prime power $q$, choose an integer~$r$ and a primitive $\theta\in\F_q$, and let $Z_{n,\theta}$ be the Sidelnikov sequence of length $n=q-1$ with respect to~$\theta$. Let $T > 0$ be real. If $t/n\to T$ as $n\to\infty$, then $F(Z_{n,\theta}^{r,t})\to h(T)$ as $n\to\infty$.
\end{conjecture}


\providecommand{\bysame}{\leavevmode\hbox to3em{\hrulefill}\thinspace}
\providecommand{\MR}{\relax\ifhmode\unskip\space\fi MR }
\providecommand{\MRhref}[2]{%
  \href{http://www.ams.org/mathscinet-getitem?mr=#1}{#2}
}
\providecommand{\href}[2]{#2}

\end{document}